\documentclass[11pt,dvips,twoside,letterpaper]{article}
\usepackage{pslatex}
\usepackage{fancyhdr}
\usepackage{graphicx}
\usepackage{geometry}
\usepackage{url}

\RequirePackage{amsfonts,amssymb,amsmath,amscd,amsthm}
\RequirePackage{txfonts}
\RequirePackage{graphicx}
\RequirePackage{xcolor}
\RequirePackage{geometry}
\RequirePackage{enumerate}

\def\figurename{Figure} 
\makeatletter
\renewcommand{\fnum@figure}[1]{\figurename~\thefigure.}
\makeatother

\def\tablename{Table} 
\makeatletter
\renewcommand{\fnum@table}[1]{\tablename~\thetable.}
\makeatother

\usepackage{amsmath}
\usepackage{amssymb}
\usepackage{amsfonts}
\usepackage{amsthm,amscd}

\newtheorem{theorem}{Theorem}[section]

\newtheorem{proposition}[theorem]{Proposition}
\theoremstyle{definition}
\newtheorem{definition}[theorem]{Definition}

\theoremstyle{remark}

\numberwithin{equation}{section}

\def\R{\mathbb R}

\begin{document}
\title{\bfseries\scshape{Point and Potential Symmetries of the Fokker Planck Equation}}
\author{\bfseries\scshape Faya Doumbo Kamano\thanks{E-mail address : doumbokamano@yahoo.fr} \\
Institut Sup\'erieur de Formation \`a Distance (ISFAD),\\
Conakry, {\bf Guin\'ee}. 
\\\bfseries\scshape Bakary Manga\thanks{E-mail address : bakary.manga@ucad.edu.sn, bakary.manga@imsp-uac.org} \\
D\'epartement de Math\'ematiques et Informatique, \\
Universit\'e Cheikh Anta Diop de Dakar,  {\bf S\'en\'egal}.  \\
and \\
Institut de Math\'ematiques et de Sciences Physiques (IMSP), \\
01 BP 613, Porto-Novo, {\bf B\'enin}.
\\\bfseries\scshape Jo\"el Tossa\thanks{E-mail address :  joel.tossa@imsp-uac.org} \\
Institut de Math\'ematiques et de Sciences Physiques (IMSP) \\
Universit\'e d'Abomey-Calavi, {\bf B\'enin}.
}

\date{}
\maketitle \thispagestyle{empty} \setcounter{page}{1}

\fancyfoot{}
\renewcommand{\headrulewidth}{0pt}

\begin{abstract}
We determine the Lie point symmetries of the Fokker-Planck equation and  provide 
examples of solutions of this equation. The Fokker-Planck equation admits a conserved form, 
hence there is an  auxiliary system associated to this equation and whose point symmetries give 
rise to potential  symmetries of the Fokker-Planck equation. We therefore use those potential symmetries to provide other solutions of the 
Fokker-Planck equation. 
\end{abstract}



\section{Introduction}

\noindent 
The Fokker-Planck equation (FPE, for short) is a linear PDE that describes  
the transition probability density of a Markov process. It is also known as the Kolmogorov diffusion equation and 
is used to model many situations such as  evolution of the distribution 
function  of a particle, finance, turbulence, population dynamics, protein 
kinetics (see \cite{carrillo-cordier-mancini:decision-making-fp}, 
\cite{carrillo-cordier-mancini:one-dimensional-fp}, \cite{gardiner-handbook}, 
\cite{martin:processus-Stochastiques}, \cite{till:nonlinear-fp}).
The FPE interests many researchers as shown by the number of publications
on the subject; see e.g. \cite{brics-kaupzs-mahnke}, \cite{carrillo-cordier-mancini:decision-making-fp,
carrillo-cordier-mancini:one-dimensional-fp}, \cite{chancelier-cohen-pacard}, \cite{hesam-nazemi-haghbin}, \cite{hottovy},
\cite{ouhadan-kinani-rahmoune-awane}, \cite{ouhadan-kinani-rahmoune-awane-ammar-essabab}, \cite{tanski}, 
\cite{till:nonlinear-fp} and references therein.

We state the Fokker-Planck equation (FPE) in the following form:
\begin{eqnarray}{\label{e01}}
u_t(x,t)=-a_2u(x,t)-(a_2x+a_1)u_x(x,t)+\frac{1}{2}u_{xx}(x,t), 
\end{eqnarray}
where $a_1$ and $a_2$ are  real numbers; $u(x,t)$ is a function that depends on the variables 
$x$ and $t$, to be determined; and  $u_\alpha$ denotes differentiation of $u$ with respect 
to the variable $\alpha$. Like most PDEs, it gives explicit solutions only in very specific cases 
related both to the form of the equation and the shape of the area where it is studied.
Many techniques are used to solve particular cases of the FPE: 
quantum mechanics technique (\cite{brics-kaupzs-mahnke}), Fourier transform method (\cite{tanski}),
differential transform method (\cite{hesam-nazemi-haghbin}), numerical method (e.g. 
\cite{carrillo-cordier-mancini:decision-making-fp, carrillo-cordier-mancini:one-dimensional-fp}, 
 \cite{hottovy}, \cite{zorzano-mais-vazquez}).

Powerful means  used in the study of DEs and PDEs are the Lie symmetries. 
Since their introduction by Sophus Lie (\cite{lie-theor-transf}), 
Lie symmetries are experiencing a rapid development  as a wonderful tool 
for the classification  of invariant solutions of DEs and PDEs.
Point symmetries are local symmetries as their infinitesimals depend on independent variables $x$'s, dependent variables $u(x)$'s, 
and derivatives of dependent variables; and are determined if $u(x)$ is sufficiently smooth in some neighbourhood of $x$.
Potential symmetries when with them are non-local symmetries whose infinitesimals, at any point $x$, depend on the global behavior of 
$u(x)$. Potential symmetries are  very useful as they lead to the construction of solutions of  a given system of PDEs which cannot be
obtained as invariant solutions of its local symmetries. See  Section \ref{sec:potential-symmetries} for wider discussion on potential symmetries.
See also {\it Chap 7} of \cite{bluman-kumei:symmetries-de} for more about potential symmetries. 

The FPE is considered in \cite{ouhadan-kinani-rahmoune-awane} and in \cite{pucci-saccomandi} in the form $u_t=u+xu_x+u_{xx}$ 
which is different from (\ref{e01}).  
The authors of the papers quoted above have determined the Lie point symmetries of the FPE, as well as the potential symmetries.  
They also have provided families of solutions of the FPE.
%

In this paper we consider the FPE (\ref{e01}) with the condition  $a_2\neq 0$.  
We adopt the same approach as in \cite{ouhadan-kinani-rahmoune-awane} and determine 
the Lie point symmetries of the FPE in Section \ref{sec:point-symmetries}. 
Some of its solutions are also determined. In Section \ref{sec:potential-symmetries}, 
we show that the FPE can be written in a conserved form. A conserved form leads to auxiliary dependent variables (which are potentials) and 
then to an auxiliary system of PDEs whose local symmetries are the potential symmetries of the FPE. We determine such symmetries in 
Section \ref{sec:potential-symmetries} and use them to construct other solutions of the Fokker-Planck equation.

\section{Point symmetries of the Fokker-Planck equation \label{sec:point-symmetries}}
\subsection{Some basics about Lie point symmetries}

Consider a general system of $n^{th}$ order DEs admitting $p$ independent variables $x = (x_1,\ldots,x_p)$ in 
$X\simeq \mathbb{R}^p$ and $q$ dependent variables $u = (u_1,\dots,u_q)$ in $U\simeq \mathbb R^q$,
\begin{eqnarray}\label{sde}
\Delta_{\nu} (x, u^{(n)}) = 0, \quad \nu = 1,\ldots,m,
\end{eqnarray}
with $u^{(n)}$ denoting the derivatives of the $u$'s with respect to the $x$'s up to order $n$. The
system (\ref{sde}) is thus defined by the vanishing of a collection of differentiable functions 
$\Delta_\nu: J^n \to \mathbb{R}$ defined on the $n^{th}$ jet space $J^{n}=J^{n}E=X\times U^{(n)}$,
where $E$ is the {\it total space} $E=X\times U$ (see \cite{olver-eq-inv-sym}). The points in the
vertical space  $U^{(n)}$ are denoted by $u^{(n)}$ and consist of all the dependent variables
and their derivatives up to order $n$. The system (\ref{sde}) can therefore be viewed as defining (or defined by) a variety
$
S_\Delta =\{(x,u^{(n)})/ \Delta_\nu (x,u^{(n)}) = 0, \, \nu = 1,\ldots,m\},
$ 
contained in the $n^{th}$ order jet space, and consisting of all points $(x,u^{(n)})\in J^n$ satisfying the
system. The defining functions $\Delta_{\nu}$ are assumed to be regular in a neighbourhood of $S_{\Delta}$; in particular, 
this is the case if the Jacobian matrix of the functions  $\Delta_{\nu}$ with respect to the jet variables $(x, u^{(n)})$ has maximal 
rank $m$ everywhere on $S_\Delta$. In the case of point transformations, the infinitesimal generators
form a Lie algebra $\mathcal{G}$ consisting of vector fields
$
 \displaystyle V=\sum^p_{i=1}\xi^i(x,u)\frac{\partial}{\partial x^i}+\sum^q_{\alpha=1}\eta^\alpha(x,u)\frac{\partial}{\partial u^\alpha} 
$ 
on the space of independent and dependent variables. Let $V^{(n)}$ denote the $n^{th}$ prolongation  
of  $V$ to the jet space $J^n$ (\cite[p.117]{olver-eq-inv-sym}): 
\vspace{-0.2cm}
\begin{equation}
 V^{(n)}=\sum^p_{i=1}\xi^i(x,u)\frac{\partial}{\partial x^i}+
\sum^q_{\alpha=1}\sum^n_{\# J=j=0}\eta_J^\alpha(x,u^{(j)})\frac{\partial}{\partial u^\alpha_J}, 
\end{equation}
for any unordered multi-index $J=(j_1,\dots,j_k), 1\leqslant j_k\leqslant p$ of order $k=\# J=j_1+\cdots+j_k\leqslant n$; where, 
for any $\alpha=1,\dots,q$, 
\begin{equation}
\eta_J^\alpha=D_JQ^\alpha+\sum_{i=1}^p\xi^iu_{J,i}^\alpha  \text{ and } 
Q^\alpha=\eta^\alpha(x,u)-\sum_{i=1}^p\xi^i(x,u)\frac{\partial u^\alpha}{\partial x^i}.
\end{equation}

The fundamental infinitesimal symmetry criterion for the system (\ref{sde}) is stated in the
\begin{theorem}[\cite{olver-eq-inv-sym}]\label{ict}
 A connected group of transformations $G$ is a symmetry group of the
fully regular system of DEs (\ref{sde}) if and only if the infinitesimal symmetry conditions
\begin{eqnarray}{\label{ic}}
V^{(n)} (\Delta_\nu ) = 0, \; \nu = 1,\dots, m, \text{ whenever } \Delta = 0,
\end{eqnarray}
hold for every infinitesimal generator $V$ of the Lie algebra $\mathcal{G}$ of $G$.
\end{theorem}
Let $u=f(x)=f(x_1,\dots,x_p)$ be a function of $\mathbb{R}^p$ with values in $\mathbb{R}$. It is known that there exists 
$p_r=\binom{p+r-1}{p}$ derivatives of $f$ of order $r$. The equation $\Delta(x,u^{(n)})=0$  is defined 
on the space $\mathbb{R}^p\times U^{n}$ of dimension $p+qp^{(n)}$, with  
 $p^{(n)}=1+p_1+p_2+\cdots+p_n=\binom{p+n}{n}$. A system $\Delta(x,u^{(n)})=(\Delta_1(x,u^{(n)}),\dots,\Delta_m(x,u^{(n)}))$
 will have as Jacobian matrix, a matrix of rank $m\times (p+qp^{(n)})$. 
See more details in \cite[p. 95]{olver-ALGDE}.
\begin{definition}[\cite{olver-ALGDE}]
The system (\ref{sde}) 
is said to be of maximum rank
if the $m\times (p+qp^{(n)})$ Jacobian matrix $J_\Delta(x,u^{(n)}):=\left(\frac{\partial \Delta_\nu}{\partial x^i}, 
\frac{\partial \Delta_\nu}{\partial u_J^\alpha}\right)$ of $\Delta$ with respect to all the variables $(x,u^{(n)})$ is of
rank $m$ whenever $\Delta(x,u^{(n)})=0$.
\end{definition}

\subsection{Lie point symmetries of the FPE}

To investigate the Lie point symmetries of the FPE, we have to check the maximal rank condition for the map 
$\Delta: (x,t;u^{(2)})\mapsto u_t(x,t)+a_2u(x,t)+(a_2x+a_1)u_x(x,t)-\frac{1}{2}u_{xx}(x,t)$ whose kernel equation is  (\ref{e01})
on a subset $M^{(2)}$ of the $2^{nd}$ jet-space $X\times U^{(2)}$ of the manifold $X\times U$. 
The independent variables $(x,t)$ 
and the dependent variable $u$ leave on the spaces  $X\simeq \mathbb R^2$ and $U\simeq \mathbb R$, respectively. 
The expression $u^{(2)}=(u,u_x,u_t,u_{xx},u_{xt},u_{tt})$ 
represents the various partial derivatives up to the second order of $u$, and leaves on the second prolongation $U^{(2)}$ of the set $U$. 
The set  $M^{(2)}$ is the corresponding 
$2^{nd}$ prolongation of the subspace $M\subset X\times U$. The Jacobian matrix of $\Delta$,
 $J_\Delta(x,t;u^{(2)})=\left(a_2u_x,0,a_2,a_1+a_2x,1,-\frac{1}{2},0,0\right)$
does not vanish anywhere on $M^{(2)}$. Then, $\Delta$ is of maximal rank. Let 
 $V=\xi(x,t,u)\frac{\partial}{\partial x}+\tau(x,t,u)\frac{\partial}{\partial t}+\eta(x,t,u)\frac{\partial}{\partial u}$ 
be a vector field on $X\times M$,  where $\xi$, $\tau$ and $\eta$ are smooth functions.
 The second prolongation of $V$ reads 
\begin{eqnarray}{\label{e03}}
 V^{(2)}&=&V+\eta^x(x,t,u^{(2)})\frac{\partial}{\partial u_x}+\eta^t(x,t,u^{(2)})\frac{\partial}{\partial u_t}+
\eta^{xx}(x,t,u^{(2)})\frac{\partial}{\partial u_{xx}} \cr \cr
&&+\eta^{xt}(x,t,u^{(2)})\frac{\partial}{\partial u_{xt}}+
\eta^{tt}(x,t,u^{(2)})\frac{\partial}{\partial u_{tt}},
\end{eqnarray}
where $\eta^x,\eta^t,\eta^{xx},\eta^{xt}$ and $\eta^{tt}$ are given by the formulae (see \cite{olver-eq-inv-sym}):
\begin{eqnarray}
\eta^x&=&D_x(\eta-\xi u_x-\tau u_t)+\xi u_{xx}+\tau u_{xt}, \label{eq:case199}\\
\eta^t&=&D_t(\eta-\xi u_x-\tau u_t)+\xi u_{tx}+\tau u_{tt}, \label{eq:case200}\\
\eta^{xx}&=&D_{xx}(\eta-\xi u_x-\tau u_t)+\xi u_{xxx}+\tau u_{xxt}, \label{eq:case201}\\
\eta^{xt}&=&D_{xt}(\eta-\xi u_x-\tau u_t)+\xi u_{xxt}+\tau u_{ttx}, \label{eq:case202}\\
\eta^{tt}&=&D_{tt}(\eta-\xi u_x-\tau u_t)+\xi u_{ttx}+\tau u_{ttt}.\label{eq:case203}
\end{eqnarray}
\begin{proposition}\label{symmetries-of-FK}
Point symmetries of the FPE are generated by the operators
\begin{eqnarray}
V_1&=&e^{a_2t}\frac{\partial}{\partial x}, \qquad V_2=u\frac{\partial}{\partial u}, \qquad  V_4=\frac{\partial}{\partial t},\cr 
V_3&=&\frac{1}{2a_2}e^{-a_2t}\frac{\partial}{\partial x}+(a_2x+a_1)\frac{u}{a_2}e^{-a_2t}
\frac{\partial}{\partial u},\cr 
V_5&=&e^{-2a_2t}\frac{\partial}{\partial t}-(a_2x+a_1)e^{-2a_2t}\frac{\partial}{\partial x}-
2(a_2x+a_1)^2e^{-2a_2t}u\frac{\partial}{\partial u},\cr 
V_6&=&e^{2a_2t}\frac{\partial}{\partial t}+(a_2x+a_1)e^{2a_2t}\frac{\partial}{\partial x}
-a_2ue^{2a_2t}\frac{\partial}{\partial u},\nonumber
\end{eqnarray}
and an infinite number of generators 
 $V_\alpha=\alpha(x,t)\frac{\partial}{\partial u}$; 
where $\alpha$ is any solution  of the FPE.
\end{proposition}
\begin{proof}
We make the assumption $V^{(2)}\Delta(x,t;u^{(2)})=0$ whenever $\Delta(x,t;u^{(2)})=0$, 
and check the corresponding conditions on $\xi$, $\tau$ and $\eta$. Those conditions lead to
\begin{eqnarray}\label{e05}
\left(\eta^t=-a_2\eta-a_1\eta^x-a_2\xi u_x-a_2 x\eta^x+\frac{1}{2}\eta^{xx}\right)_{\mid\Delta=0}.
\end{eqnarray}
Now replace $\eta^x$, $\eta^t$ and $\eta^{xx}$ in (\ref{e05}) by their expressions given in 
(\ref{eq:case199}), (\ref{eq:case200}) and (\ref{eq:case201}) respectively, and  
eliminate $u_t$ by substituting it by the right hand side of (\ref{e01}) any time when it occurs. 
Then the derivatives of $u$ with right to $t$ disappear.
So, the resolution of the corresponding system of PDEs is equivalent to solving the following system:
\begin{eqnarray}
\eta_{uu}=0, &\label{eq:case213}\\
2\xi_x-\tau_t=0,  \qquad \tau_x=0, \qquad \tau_u=0,\qquad \xi_u=0, \label{eq:case211}\\
2(a_2x+a_1)\xi_x-2\xi_t+2a_2\xi-2\eta_{xu}=0, &\label{eq:case209}  \\
\eta_{xx}-2(a_2x+a_1)\eta_x-2\eta_t+2a_2u\eta_u-2a_2\eta-2a_2u\tau_t=0, &\label{eq:case207}
\end{eqnarray}
Equation (\ref{eq:case213}) implies that $\eta$ is linear in $u$. So, it writes
\begin{eqnarray}
 \eta(x,t,u)=A(x,t)u+B(x,t),
\end{eqnarray}
$A$ and $B$ being smooth functions depending only on $x$ and $t$. From (\ref{eq:case211}), 
we get 
\begin{eqnarray}
 \xi=\frac{1}{2}\tau_tx+k(t),
\end{eqnarray}
where $k$  is a smooth function of $t$. Substituting $\xi$ and $\eta$ by 
their expressions in (\ref{eq:case209}) and 
differentiating the resulting expression with respect to $x$, we get  
$2A_{xx}-2a_2\tau_t+\tau_{tt}=0$. Thus,
\begin{eqnarray}
& & A(x,t)=\left(\frac{1}{2}a_2\tau_t-\frac{1}{4}\tau_{tt}\right)x^2+A_1(t)x+A_2(t),\label{eq:case215}\\ 
& & k'(t)-a_2k(t)+A_1(t)-\frac{1}{2}a_1\tau_t=0, \label{eq:case216}
\end{eqnarray}
where $A_1$ and $A_2$ are smooth functions of the variable $t$. Using Equation (\ref{eq:case207}), we find that
\begin{eqnarray}
A_{xx}-2a_2xA_x-2a_1A_x-2A_t-2a_2\tau_t&=&0, \label{eq:case217}\\
-\frac{1}{2}B_{xx}+a_2xB_x+a_1B_x+a_2B+B_t&=&0. \label{eq:case218}
\end{eqnarray}
Note that (\ref{eq:case218}) is nothing but the FPE (\ref{e01}). Now (\ref{eq:case215}), (\ref{eq:case216}) and (\ref{eq:case217}) entail
\begin{eqnarray}
\tau(t)&=&C_1e^{2a_2t}+C_2e^{-2a_2t}+C_3, \cr
A_1(t)&=&[-4a_1a_2C_2e^{-a_2t}+C_4]e^{-a_2t}, \cr
A_2(t)&=&-2a_1^2C_2e^{-2a_2t}+\frac{1}{a_2}C_4a_1e^{-a_2t}-C_1a_2e^{2a_2t}+C_5, \cr 
k(t)&=&a_1C_1e^{a_2t}-a_1C_2e^{-2a_2t}+\frac{C_4}{2a_2}e^{-a_2t}+C_6e^{a_2t},\nonumber
\end{eqnarray}
where $C_1, C_2,\dots,C_6$ are real numbers.
Hence, the solution of the system (\ref{eq:case207})-(\ref{eq:case213}) is
\begin{eqnarray}
\xi(x,t,u)&=&[C_1a_2e^{2a_2t}-C_2a_2e^{-2a_2t}]x+a_1C_1e^{2a_2t}-a_1C_2e^{-2a_2t}+\frac{C_4}{2a_2}e^{-a_2t}+C_6e^{a_2t}, \cr
\tau(x,t,u)&=&C_1e^{2a_2t}+C_2e^{-2a_2t}+C_3, \cr
\eta(x,t,u)&=&\left(-2C_2(a_2x+a_1)^2e^{-2a_2t}+\frac{C_4}{a_2}(a_2x+a_1)e^{-a_2t}\right)u  
            -C_1a_2e^{2a_2t}u+C_5u+\alpha(x,t), \nonumber
\end{eqnarray}
where $\alpha(x,t)=B(x,t)$ is any solution of the FPE. The rest of the proof is straightforward.
\end{proof}

\subsection{Examples of solutions of the FPE}

In the sequel, we provide a family of solutions of the Fokker-Planck equation (\ref{e01}).
\begin{theorem}{\label{solutions-of-FP}}
Let $\alpha(x,t)$ be any solution of the FPE. Then the  functions  
\begin{eqnarray}
f_1(x,t)&=&e^{-2a_2t}\Big[\alpha_t-(a_2x+a_1)\alpha_x+2(a_2x+a_1)^2\alpha\Big], \label{eq:case219}\\
f_2(x,t)&=&\frac{e^{-a_2t}}{a_2}\left(\frac{1}{2}\alpha_x-(a_2x+a_1)\alpha\right),  \label{eq:case220} \\
f_3(x,t)&=&e^{2a_2t}\Big[\alpha_t+(a_2x+a_1)\alpha_x+a_2\alpha\Big], \label{eq:case221} \\
f_4(x,t)&=&\alpha_xe^{a_2t}, \qquad f_5(x,t)=\alpha_t\label{eq:case222} 
\end{eqnarray}
are also solutions  of the FPE.
\end{theorem}
\begin{proof}
Since $\{V_\alpha, V_i,\, i=1,\dots,6\}$  generates a Lie algebra, 
the  stability of the brackets in the table below completes the proof.
\begin{table}[h!]
\begin{center}
\begin{tabular}{|l|l|l|}
\hline
$[V_1,V_2]=0$          & $[V_2,V_4]=0$      & $[V_3,V_\alpha]=V_{\frac{e^{-a_2t}}{a_2}
                                           \left[\frac{1}{2}\alpha_x-(a_2x+a_1)\alpha\right]}$    \cr
$[V_1,V_3]=V_2$        & $[V_2,V_5]=0$      & $[V_4,V_5]=-2a_2V_5 $                               \cr                      
$[V_1,V_4]=-a_2V_1$    & $[V_2,V_6]=0$      & $[V_4,V_6]=2a_2V_6 $                                \cr
$[V_1,V_5]=-4a_2^2V_3$ & $[V_2,V_\alpha]=-V_\alpha$ & $[V_4,V_\alpha]=V_{\alpha_t}$               \cr
$[V_1,V_6]=0$          & $[V_3,V_4]=a_2V_3$ & $[V_5,V_6]=4a_2V_4-2a_2^2V_2$                       \cr
$[V_1,V_\alpha]=V_{\alpha_xe^{a_2t}}$ & $[V_3,V_5]=0 $ & $[V_5,V_\alpha]=V_{e^{-2a_2t}[\alpha_t-
                                                           (a_2x+a_1)\alpha_x+2(a_2x+a_1)^2\alpha]}$ \cr
$[V_2,V_3]=0 $        & $[V_3,V_6]=V_1$ & $[V_6,V_\alpha]=V_{e^{2a_2t}[\alpha_t+(a_2x+a_1)\alpha_x
                                          +a_2\alpha]}$ \cr
\hline
\end{tabular}
\caption{\label{tab:brackets} Commutations table of the Lie algebra of symmetries of the FPE}
\end{center}
\end{table}
%

\end{proof}
As mentioned in \cite{ouhadan-kinani-rahmoune-awane}, using the Lie brackets in Table \ref{tab:brackets}, one 
can construct a family of solutions from a trivial solution.
Consider  e.g. $u(x,t)=e^{-a_2t}$, then the functions
\begin{eqnarray}
g_1(x,t)&=&[-a_2+2(a_2x+a_1)^2]e^{-3a_2t}, \\
g_2(x,t)&=&-\left(x+\frac{a_1}{a_2}\right)e^{-2a_2t},\\
g_3(x,t)&=&-a_2e^{-a_2t}
\end{eqnarray}
are also solutions of (\ref{e01}). From these solutions we can again construct other solutions. 
For instance, applying the symmetry generators (\ref{eq:case220}) to $g_1$ yields  to the solution
\begin{eqnarray}
 g_4(x,t)=\left[3a_2^2-12a_2(a_2x+a_1)^2+4(a_2x+a_1)^4\right]\exp(-5a_2t).
\end{eqnarray}

\section{Potential symmetries of the FPE \label{sec:potential-symmetries}}
\subsection{Preliminaries on potential symmetries}
A partial differential equation of order $n$ in the unknown function $u(x,t)$
\begin{eqnarray}{\label{pde}}
  \Delta(x,t,u^{(n)})=0,
\end{eqnarray}
is written in a conserved form if it has the following form:
\begin{eqnarray}{\label{cf}}
  D_tT(x,t,u^{(n-1)})+D_xX(x,t,u^{(n-1)})=0.
\end{eqnarray}
Since the PDE (\ref{cf}) is in a conserved form, a potential $v$ considered as a new variable is introduced. 
A system of PDEs denoted by $\mathcal{S}(x,t,u^{(n-1)},v_x,v_t)$ is then obtained. If $(u(x,t),v(x,t))$ is a solution of 
the system  $\mathcal{S}(x,t,u^{(n-1)},v_x,v_t)$,
then $u(x,t)$ solves the PDE  given by (\ref{pde}).
\begin{definition}
 Assume that the auxiliary system $\mathcal{S}(x,t,u^{(n-1)},v_x,v_t)$ admits a generator  $W$ of point symmetries given by
$
W=\xi(x,t,u,v)\frac{\partial}{\partial x}+\tau(x,t,u,v)\frac{\partial}{\partial t}+\eta(x,t,u,v)\frac{\partial}{\partial u}+
\phi(x,t,u,v)\frac{\partial}{\partial v}.
$ 
One says that  $\mathcal{S}(x,t,u^{(n-1)},v_x,v_t)$ defines a potential symmetry admitted by (\ref{pde}) if and only if one, at least,
of the infinitesimals $\xi$, $\tau$ and $\eta$ depends explicitly on the potential $v$;  that is if and only if the condition 
\begin{eqnarray}{\label{eq:potential-condition}}
\left(\frac{\partial \xi}{\partial v}\right)^2+\left(\frac{\partial \tau}{\partial v}\right)^2+
\left(\frac{\partial \eta}{\partial v}\right)^2\neq 0 
\end{eqnarray}
holds. In this case, the symmetry $Y=\xi(x,t,u,v)\frac{\partial}{\partial x}+\tau(x,t,u,v)\frac{\partial}{\partial t}
+\eta(x,t,u,v)\frac{\partial}{\partial u}$ will be called a potential symmetry of Equation (\ref{pde}).
\end{definition}
Potential symmetries can also be used in the study of a boundary value problem posed for a given system of PDEs and for 
the study of ODEs. For a scalar ODE, a potential symmetry reduces the order (see \cite{bluman-kumei:symmetries-de}).

We are now going to explain how, from potential symmetries, one obtains solutions of the PDE (\ref{pde}) which admits 
a conserved form (\ref{cf}). See \cite{pucci-saccomandi} for wider discussion. 
Given a point symmetry $\xi\frac{\partial}{\partial x}+\tau\frac{\partial}{\partial t}
+\eta\frac{\partial}{\partial u}+\phi\frac{\partial}{\partial v}$ of (\ref{cf}), the invariant surface conditions are
\begin{eqnarray}
 \xi(x,t,u,v)u_x+\tau(x,t,u,v)u_t-\eta(x,t,u,v)=0,  \label{eq:surface-condition1}\\
 \xi(x,t,u,v)v_x+\tau(x,t,u,v)v_t-\phi(x,t,u,v)=0.  \label{eq:surface-condition2}
\end{eqnarray}
The associated characteristic system yields to the following independent integrals
\begin{equation}
 s_1(x,t,u,v)=c_1, \quad s_2(x,t,u,v)=c_2, \quad s_3(x,t,u,v)=c_3, \label{eq:integrals}
\end{equation}
with $\dfrac{\partial (s_1,s_2,s_3)}{\partial(u,v)}$ of rank $2$. 
If we set $z=c_1$, $c_2=h_1(z)$ and $c_3=h_2(z)$, we obtain from (\ref{eq:integrals}):
\begin{eqnarray}
u=U(x,t,z,h_1(z),h_2(z)), \label{eq:solution1-surface-condition}\\
v=V(x,t,z,h_1(z),h_2(z)), \label{eq:solution2-surface-condition}\\
G(x,t,z,h_1(z),h_2(z))=0.
\end{eqnarray}
The invariant solutions of (\ref{cf}) are given by (\ref{eq:solution1-surface-condition}) and (\ref{eq:solution2-surface-condition}),
where $h_i(z)$ are the solutions of the ordinary system obtained by substitution in (\ref{cf}). Since (\ref{pde}) is a differential
consequence of (\ref{cf}), the solution of (\ref{cf}) give those solutions of (\ref{pde}), which verify the differential relation 
obtained by eliminating $v$ between (\ref{eq:surface-condition1}) and 
$ \xi T+\tau X-\phi=0$.

\subsection{Potential symmetries of the FPE}

The conserved form of the FPE  can be written as
$D_tu+D_x\left(-(a_2x+a_1)u+\frac{1}{2}u_x\right)=0$.
Then, the corresponding system writes as follows:
\begin{equation}\left\{
\begin{array}{ccll}
v_t&=&-(a_2x+a_1)u+\frac{1}{2}u_x, & \cr
v_x&=&u, &  
\end{array}\right. \label{e224}
\end{equation}
where the potential variable $v$ has been introduced as a new dependent variable.
\begin{proposition}\label{prop:symmetries-auxiliary-system}
 The system (\ref{e224}), with $a_1\in \R$ and $a_2\neq 0$,  admits a non trivial 
 symmetry group with the following  infinitesimal generators:
\begin{eqnarray}
W_1&=&e^{a_2t}\frac{\partial}{\partial x}, \qquad W_2=u\frac{\partial}{\partial u}+v\frac{\partial}{\partial v}, \qquad 
W_4=\frac{\partial}{\partial t},\cr \cr 
W_3&=&\frac{1}{2a_2}e^{-a_2t}\frac{\partial}{\partial x}+\left[\left(x+\frac{a_1}{a_2}\right)u+v\right]e^{-a_2t}\frac{\partial}{\partial u}
+\left(\frac{a_1}{a_2}v+xv\right)e^{-a_2t}\frac{\partial}{\partial v}, \cr \cr
W_5&=&-(a_2x+a_1)e^{-2a_2t}\frac{\partial}{\partial x}+e^{-2a_2t}\frac{\partial}{\partial t}
-2\left[\Big((a_2x+a_1)^2-a_2\Big)u+2a_2(a_2x+a_1)v\right]e^{-2a_2t}\frac{\partial}{\partial u} \cr
& &-2\left[(a_2x+a_1)^2-a_2\right]ve^{-2a_2t}\frac{\partial}{\partial v}, \cr\cr
W_6&=&e^{2a_2t}\frac{\partial}{\partial t}+\left(a_2x+a_1\right)e^{2a_2t}\frac{\partial}{\partial x}-
a_2ue^{2a_2t}\frac{\partial}{\partial u}, \nonumber
\end{eqnarray}
and an infinite number of generators of the form
$W_\beta=\beta_x(x,t)\frac{\partial}{\partial u}+\beta(x,t)\frac{\partial}{\partial v}$, 
where $\beta(x,t)$ satisfies the equation
$
 \beta_t=-(a_2x+a_1)\beta_x+\frac{1}{2}\beta_{xx}.
$
\end{proposition}
\begin{proof}
Let $\Delta_{1}=v_t+(a_2x+a_1)u-\frac{1}{2}u_x$ and $\Delta_{2}=v_{x}-u$ 
 be the associated system to the system  (\ref{e224}) and let
 $W=\xi(x,t,u,v)\frac{\partial}{\partial x}+\tau(x,t,u,v)\frac{\partial}{\partial t}+\eta(x,t,u,v)\frac{\partial}{\partial u}+
 \phi(x,t,u,v)\frac{\partial}{\partial v}$ be a symmetry vector field of this system. 
The  criterion (\ref{ic})  writes
$
 W^{(2)}(\Delta_{i})|_{\Delta_{i}=0,i=1,2}=0,
$ 
where 
\begin{eqnarray}
 W^{(2)}&=&W+\eta^{x}\frac{\partial}{\partial u_{x}}+\eta^{t}\frac{\partial}{\partial u_{t}}+\phi^{x}\frac{\partial}{\partial v_{x}}+
\phi^{t}\frac{\partial}{\partial v_{t}}+\eta^{xx}\frac{\partial}{\partial u_{xx}}+\eta^{xt}\frac{\partial}{\partial u_{xt}} \cr \cr 
&&+\eta^{tt}\frac{\partial}{\partial u_{tt}}+\phi^{xx}\frac{\partial}{\partial v_{xx}}+\phi^{xt}\frac{\partial}{\partial v_{xt}}+
\phi^{tt}\frac{\partial}{\partial v_{tt}}.
\end{eqnarray}
The coefficient functions $(\eta^{x},\eta^{t},\phi^{x},\phi^{t},\eta^{xx},\eta^{xt},\eta^{tt},\phi^{xx},\phi^{xt},\phi^{tt})$
in $W^{(2)}$ are given as follows
\begin{eqnarray}
\eta^{x}&=&D_{x}\eta-u_{x}D_{x}\xi-u_{t}D_{x}\tau, \quad \eta^{t}=D_{t}\eta-u_{x}D_{t}\xi-u_{t}D_{t}\tau,\cr \cr
\phi^{x}&=&D_{x}\phi-v_{x}D_{x}\xi-v_{t}D_{x}\tau, \quad \phi^{t}=D_{t}\phi-v_{x}D_{t}\xi-v_{t}D_{t}\tau,\cr \cr
\eta^{xx}&=&D^{2}_{x}\eta-u_{x}D^{2}_{x}\xi-u_{t}D^{2}_{x}\tau-2u_{xx}D_{x}\xi-2u_{xt}D_{x}\tau,\cr \cr
\phi^{xx}&=&D^{2}_{x}\phi-v_{x}D^{2}_{x}\xi-v_{t}D^{2}_{x}\tau-2v_{xx}D_{x}\xi-2v_{xt}D_{x}\tau,\cr \cr
\eta^{xt}&=&D^{2}_{xt}\eta-u_{x}D^{2}_{xt}\xi-u_{t}D^{2}_{xt}\tau-u_{xx}D_{t}\xi-u_{tx}D_{t}\tau-u_{xt}D_{x}\xi-u_{tt}D_{x}\tau, \cr \cr
\phi^{xt}&=&D^{2}_{xt}\phi-v_{x}D^{2}_{xt}\xi-v_{t}D^{2}_{xt}\tau-v_{xx}D_{t}\xi-v_{tx}D_{t}\tau-v_{xt}D_{x}\xi-v_{tt}D_{x}\tau, \cr \cr
\eta^{tt}&=&D^{2}_{t}\eta-u_{x}D^{2}_{t}\xi-u_{t}D^{2}_{t}\tau-2u_{xt}D_{t}\xi-2u_{tt}D_{t}\tau,\cr \cr
\phi^{tt}&=&D^{2}_{t}\phi-v_{x}D^{2}_{t}\xi-v_{t}D^{2}_{t}\tau-2v_{xt}D_{t}\xi-2v_{tt}D_{t}\tau,\nonumber
\end{eqnarray}
Hence, the criterion $W^{(2)}(\Delta_{i})|_{\Delta_{i}=0,i=1,2}=0$ gives the following equalities:
\begin{eqnarray}
\Big(\phi^{t}+a_2\xi u+(a_2x+a_1)\eta-\frac{1}{2}\eta^x=0\Big)_{|{\Delta_{i}=0,i=1,2}} {\label{e225}}\\ 
\Big(\phi^{x}-\eta=0\Big)_{|{\Delta_{i}=0,i=1,2}}
\end{eqnarray}
 Replacing $\phi^{x},\phi^{t}$ and $\eta^{x}$ by their 
 expressions in (\ref{e225}) and equalizing the coefficients of the remaining unconstrained partial derivatives of $u$ and $v$ to zero, 
one obtains: 
\begin{eqnarray}
\xi_u&=&0, \quad \xi_v=0, \quad 2\xi_x=\tau_t,\label{eq:case226}\\
\tau_x&=&0, \quad \tau_u=0, \qquad \tau_v=0, \label{eq:case228}\\
\phi_u&=&0, \quad \phi_{vv}=0, \label{eq:case229}\\
\tau_{ttt}&=&4a_2^2\tau_t, \quad 2\xi_{tt}=3a_2\big[(a_2x+a_1)\tau_t+\frac{2}{3}a_2\xi\big],\label{eq:case230}\\
2\eta&=&2\phi_x-\tau_tu+2u\phi_v, \label{eq:case232}\\
2\phi_{vx}&=&(a_2x+a_1)\tau_t-2\xi_t+2a_2\xi, \label{eq:case233}\\
\phi_{xx}&=&2(a_2x+a_1)\phi_x+2\phi_t, \label{eq:case234}\\
4\phi_{tv}&=&\!\!-\tau_{tt}+\left(-2a_2^2x^2+(2-4a_1x)a_2-2a_1^2\right)\tau_t+(a_2x+a_1)\xi_t -a_2(a_2x+a_1)\xi. \label{eq:case235}
\end{eqnarray}
Equations (\ref{eq:case228}) imply that $\tau$ depends only  on $t$. 
Hence,  relations in (\ref{eq:case226}) yield to
\begin{equation}\label{e236}
 \xi=\frac{1}{2}\tau_tx+L(t),
\end{equation}
where $L$ is a smooth function of $t$. Relations (\ref{eq:case229}) imply that $\phi$ is independent from $u$ and is linear with 
right to $v$. That is there exists functions $D$ and $E$ depending only on $x$ and $t$ such that
\begin{eqnarray}{\label{e237}}
 \phi=D(x,t)v+E(x,t).
\end{eqnarray}
Then, substituting $\xi$ and $\phi$ by their expressions in 
(\ref{eq:case233}) and differentiating the resulting expression with respect to $x$, one obtains the equation 
$2D_{xx}-2a_2\tau_t+\tau_{tt}=0$. Thus
\begin{eqnarray}
& & D(x,t)=\left(\frac{1}{2}a_2\tau_t-\frac{1}{4}\tau_{tt}\right)x^2+B_1(t)x+B_2(t), \label{eq:case238}\\
& &  L'(t)-a_2L(t)+B_1(t)-\frac{1}{2}a_1\tau_t=0, \label{eq:case239} 
\end{eqnarray}
where $B_1$ and $B_2$ are smooth functions of $t$ only. Coming back  to Equation (\ref{eq:case234}), we find that
\begin{eqnarray}
D_{xx}-2(a_2x-a_1)D_x-2D_t&=&0, \label{eq:case240}\\
E_{xx}-2(a_2x-a_1)E_x-2E_t&=&0. \label{eq:case241}
\end{eqnarray}
Here again, (\ref{eq:case241}) is equivalent to (\ref{e224}).
From Equations  (\ref{eq:case238}) and (\ref{eq:case240}), one gets
\begin{eqnarray}
\tau(t)&=&C_1e^{2a_2t}+C_2e^{-2a_2t}+C_3, \cr
B_1(t)&=&(-4a_1a_2C_2e^{-a_2t}+C_4)e^{-a_2t}, \cr
B_2(t)&=&-2a_1^2C_2e^{-2a_2t}+\frac{1}{a_2}C_4a_1e^{-a_2t}+C_2a_2e^{-2a_2t}+C_5, \nonumber
\end{eqnarray}
where $C_1$, $C_2$, $C_3$, $C_4$ and $C_5$ are arbitrary constants. Now, Equation (\ref{eq:case239}) yields:
\begin{eqnarray}
 L(t)=a_1C_1e^{2a_2t}-a_1C_2e^{-2a_2t}+\frac{C_4}{2a_2}e^{-a_2t}+C_6e^{a_2t},\nonumber
\end{eqnarray}
where $C_6$ is an arbitrary constant.  Hence, expressions (\ref{eq:case232}), (\ref{e236}) and (\ref{e237})  read: 
\begin{eqnarray}
\xi&=&C_1(a_2x+a_1)e^{2a_2t} - C_2(a_2x+a_1)e^{-2a_2t} +C_4\frac{e^{-a_2t}}{2a_2}+C_6e^{a_2t},\\ \cr 
\phi&=&-2C_2\left[(a_2x+a_1)^2-a_2\right]ve^{-2a_2t}+ C_4\left(x+\frac{a_1}{a_2}\right)ve^{-a_2t}
                  +C_5v+\beta(x,t),\cr\cr 
\eta&=&-C_1a_2ue^{2a_2t}-2C_2\left[\Big((a_2x+a_1)^2-a_2\Big)u
                        +2a_2(a_2x+a_1)v\right]e^{-2a_2t}    \cr 
        & & +C_4\left[v+\left(x+\frac{a_1}{a_2}\right)u\right]e^{-a_2t}  +C_5u   +\beta_x(x,t),\nonumber
\end{eqnarray}
where $C_1,\dots,C_6$ are arbitrary constants and $\beta(x,t)=E(x,t)$ is any solution of (\ref{eq:case241}).
It is now a little matter to complete the proof.
\end{proof}
It is readily verified that $W_3$ and $W_5$ in Proposition \ref{prop:symmetries-auxiliary-system} are the only 
generators of the point symmetries of the system (\ref{e224}) that satisfy  condition (\ref{eq:potential-condition}). Hence, we have the
\begin{proposition}{\label{potential-symmetries}}
The potential symmetries of the FPE  are generated by the vector fields
%
\begin{eqnarray}
Y_1&=&\frac{1}{2a_2}e^{-a_2t}\frac{\partial}{\partial x}
+\left[\left(x+\frac{a_1}{a_2}\right)u+v\right]e^{-a_2t}\frac{\partial}{\partial u}, \cr \cr
Y_2&=&-(a_2x+a_1)e^{-2a_2t}\frac{\partial}{\partial x}+e^{-2a_2t}\frac{\partial}{\partial t}
-2\left[\Big((a_2x+a_1)^2-a_2\Big)u+2a_2(a_2x+a_1)v\right]e^{-2a_2t}\frac{\partial}{\partial u}.
\nonumber
\end{eqnarray}
\end{proposition}
%
Consider  the Symmetries $W_3$ which yields to the potential symmetry $Y_1$. The associated invariant surface conditions are 
\begin{eqnarray}
 u_x-2(a_2x+a_1)u-2a_2v &=&0, \\
 v_x-2(a_2x+a_1)v       &=&0.
\end{eqnarray}
The system below admits the following solutions: 
\begin{equation}
 u(x,t)=[2a_2xq_1(t)+q_2(t)]\exp(a_2x^2+2a_1x), \quad v(x,t)=q_1(t)\exp(a_2x^2+2a_1x), \label{eq:solution-surface-condition-FPE}
\end{equation}
where $q_1$ and $q_2$ are smooth functions of the variable $t$. If we replace the expression of $u(x,t)$ 
given by (\ref{eq:solution-surface-condition-FPE}) in (\ref{e01}), we get 
\begin{equation}
 2a_2[a_2q_1(t)+q_1'(t)]x +[q_2'(t)-q_2(t)]=0.
\end{equation}
Hence, $q_1'(t)+a_2q_1(t)=0$ and $q_2'(t)-q_2(t)=0$ and this yield to $q_1(t)=ae^{-a_2t}$ and $q_2(t)=be^t$. 
Then we have the following solution of the FPE (\ref{e01}): 
\begin{equation}
u(x,t)=[2a_2axe^{-a_2t}+be^t]\exp(a_2x^2+2a_1x), 
\end{equation}
where $a$ and $b$ are constants.

Let us now deal with the symmetry generator $W_5$ which provides the potential symmetry $Y_2$. 
The invariant surface conditions for this symmetry write 
\begin{eqnarray}
 -(a_2x+a_1)u_x+u_t+2[(a_2x+a_1)^2-a_2]u+4a_2(a_2x+a_1)v &=&0, \label{eq:invariant-surface-condition1} \\
 -(a_2x+a_1)v_x+v_t+2[(a_2x+a_1)^2-a_2]v &=&0. \label{eq:invariant-surface-condition2}
\end{eqnarray}
A solution of Equation (\ref{eq:invariant-surface-condition2}) writes
\begin{equation}
v(x, t)=\frac{f\Big((a_2x+a_1)e^{a_2t}\Big)}{(a_2x+a_1)^2}\exp\left(\frac{(a_2x+a_1)^2}{a_2}\right),
\label{eq:solution-v}
\end{equation}
where $f$ is a smooth function. Replacing the expression of $v(x,t)$ given by (\ref{eq:solution-v}) 
in (\ref{eq:invariant-surface-condition1}) and solving the latter, we get 
\begin{equation}
u(x,t)=\frac{4 a_2 xf\Big((a_2x+a_1)e^{a_2t}\Big)+g\Big((a_2x+a_1)e^{a_2t}\Big)}{(a_2x+a_1)^2}\exp\left(\frac{(a_2x+a_1)^2}{a_2}\right),
\label{eq:solution2-invariant-surface-condition}
\end{equation}
where $g$ is another smooth function. 
Now, setting $z=(a_2x+a_1)e^{a_2t}$ and putting expression (\ref{eq:solution2-invariant-surface-condition}) in 
the Fokker-Planck equation (\ref{e01}) yields to an equation that can be regarded as the vanishing of 
a polynomial of degree $3$ in $e^{a_2t}$. Then, the vanishing of the coefficients of this polynomial
leads to the following equations:
%
\begin{eqnarray}
 (1-a_2)zf'(z)+a_2f(z) &=&0, \label{eq:u-condition1}\\
 -8a_2a_1f(z)+2a_2g(z)+4za_1(a_2-1)f'(z) +z(1-a_2)g'(z) &=&0, \label{eq:u-condition2}\\
 2f(z)-2zf'(z)+z^2f''(z) &=&0,\label{eq:u-condition3} \\
 24a_1f(z)-6g(z)-16a_1zf'(z)+4zg'(z)+  4a_1z^2f''(z)-z^2g''(z)&=&0.\label{eq:u-condition4}
\end{eqnarray}
\begin{itemize}
\item If $a_2=1$, the solution of the system is $f(z)=g(z)=0$, for all $z$ and we get the trivial solution 
$u(x,t)=0$ for all $x$ and $t$.
\item Suppose $a_2\neq 1$. Then (\ref{eq:u-condition1}) gives the solution $f(z)=cz^{\frac{a_2}{a_2-1}}$, 
where $c$ is an arbitrary constant. Hence, (\ref{eq:u-condition2}) reduces to $(a-2)c=0$. 
\begin{itemize}
\item If $a_2=2$, then  the solution of the system (\ref{eq:u-condition1})-(\ref{eq:u-condition4})  is 
\begin{equation}
 f(z)=cz^2 \text{ and } g(z)=4a_1cz^2.
\end{equation}
\item If $a_2\neq 2$, then $c=0$ and $f(z)=g(z)=0$, for any $z$.
\end{itemize}
\end{itemize}
It is now clear that the potential symmetry $W_5$ yields to the solution 
\begin{equation}
u(x,t)=\lambda(2x+a_1)\exp\left(\frac{(2x+a_1)^2}{2}+4t\right),
\label{eq:solution-w5-a_2=2}
\end{equation}
for some real number $\lambda$ if $a_2=2$; and to the trivial solution $u(x,t)=0$, for all $x,t\in\R$ otherwise. 

\vskip 0.5cm
\noindent
{\bf Aknowledgement.}
The authors would like to thank the {\it Deutscher Akademischer Austaush Dienst (DAAD)} 
for its financial support. The second author is supported by the {\bf nlaga project} 
and {\bf PACER II}. Part of this work was done during his visit at IMSP (Port-Novo, Benin) funded by the 
{\it DAAD}. He  expresses his gratitude and thanks to these institutions.


%
 
\label{lastpage-01}
\end{document}